\documentclass[11pt,twoside]{amsart}
\usepackage[T1]{fontenc}
\usepackage[utf8]{inputenc}
\usepackage[english]{babel}
\usepackage{graphicx}
\usepackage{amsmath}
\usepackage{amssymb}
\usepackage{enumerate}
\usepackage[all,cmtip]{xy}
\usepackage{mathrsfs}

\usepackage{amsthm}
\usepackage{amsfonts}
\usepackage{changepage}

\newtheorem{thm1}{Theorem}[section]
\newtheorem{theorem}[thm1]{Theorem}
\newtheorem{lemma}[thm1]{Lemma}
\newtheorem{corollary}[thm1]{Corollary}
\newtheorem{proposition}[thm1]{Proposition}

\newtheorem{thmx}{Theorem}

\renewcommand{\thethmx}{\Alph{thmx}}

\theoremstyle{definition}
\newtheorem{definition}[thm1]{Definition}

\RequirePackage{ifthen}
\RequirePackage{calc}
\newcounter{exampleendflag}
\newcommand{\exendhere}{
  \setcounter{exampleendflag}{0} 
  \ifmmode
    \eqno
    \ensuremath{\blacktriangle}
  \else
    \hspace{\stretch{1}}
    \ensuremath{\blacktriangle}
  \fi
}

\newenvironment{ex}{
  \setcounter{exampleendflag}{1}
  \begin{exx}
}{
  \ifthenelse{\value{exampleendflag}=1}{\exendhere}{} 
  \end{exx}
} 
\newtheorem{exx}[thm1]{Example}

\theoremstyle{remark}
\newtheorem{remark}[thm1]{Remark}

\title{Good Hilbert functors}
\author{\gss}
\thanks{The author is supported by the Swedish Research Council, grant number 2011-5599.}

\newcommand{\saeden}{S\ae d\'en}
\newcommand{\stahl}{St\aa hl}
\newcommand{\gss}{Gustav \saeden\ \stahl}

\newcommand{\fF}{\mathcal{F}}
\newcommand{\ffF}{\mathfrak{F}}
\newcommand{\fG}{\mathcal{G}}
\newcommand{\fH}{\mathcal{H}}

\newcommand{\spec}{\operatorname{Spec}}

\newcommand{\coker}{\operatorname{coker}}

\renewcommand{\hom}{\operatorname{Hom}}
\newcommand{\ext}{\operatorname{Ext}}

\newcommand{\Coh}{\mathbf{Coh}}
\newcommand{\Cohpgs}{\mathbf{Coh}^{\operatorname{pgs}}}
\newcommand{\Cohps}{\mathbf{Coh}^{\operatorname{ps}}}
\newcommand{\QCoh}{\mathbf{QCoh}}
\renewcommand{\O}{\mathcal{O}}
\newcommand{\m}{\mathfrak{m}}

\newcommand{\Hilb}{\mathcal{H}\!\mathit{ilb}}
\newcommand{\GoodHilb}{\mathcal{H}\!\mathit{ilb}^{\operatorname{good}}}
\newcommand{\GoodHilbh}{\mathcal{H}\!\mathit{ilb}^{\operatorname{good},h}}
\newcommand{\shom}{\mathscr{H}\!\mathit{om}}
\newcommand{\Ann}{\mathscr{A}\!\mathit{nn}}
\newcommand{\End}{\mathscr{E}\!\mathit{nd}}

\newcommand{\morj}{\operatorname{Mor}^{\operatorname{flb},\operatorname{lfpb}}}
\renewcommand{\epsilon}{\varepsilon}

\newcommand{\supp}{\operatorname{Supp}}
\newcommand{\suppp}{\operatorname{Supp}^{\operatorname{good}}}
\newcommand{\too}{\hookrightarrow}

\renewcommand\O{\mathcal{O}}
\newcommand{\E}{\mathcal{E}}

\newcommand{\N}{\mathbb{N}}

\newcommand{\C}{\mathscr{C}}

\newcommand{\X}{\mathscr{X}}
\newcommand{\Y}{\mathscr{Y}}
\newcommand{\Z}{\mathscr{Z}}
\newcommand{\I}{\mathcal{I}}
\newcommand{\sch}{\mathbf{Sch}}
\newcommand{\set}{\mathbf{Set}}

\renewcommand{\m}{\mathfrak{m}}

\newcommand{\grass}{\mathcal{G}\!{\kern 0.0667em}\textit{rass}}

\newcommand{\qq}[1]{``#1''}

\newcommand{\git}{\kern-0.45em\mathbin{
  \mathchoice{/\mkern-6mu/}
    {/\mkern-6mu/}
    {/\mkern-5mu/}
    {/\mkern-5mu/}\kern-0.3em}}

\makeatletter
\newcommand{\extp}{\@ifnextchar^\@extp{\@extp^{\,}}}
\def\@extp^#1{\mathop{\bigwedge\nolimits^{\!#1}}}
\makeatother

\numberwithin{equation}{section}
\address{Department of Mathematics, KTH Royal Institute of Technology, SE-100 44 Stockholm, Sweden}
\email{gss@math.kth.se}
\begin{document}
\begin{abstract}
We introduce the good Hilbert functor and prove that it is algebraic. This functor generalizes various versions of the Hilbert moduli problem, such as the multigraded Hilbert scheme and the invariant Hilbert scheme. Moreover, we generalize a result concerning formal GAGA for good moduli spaces.
\end{abstract}
\maketitle

\section{Introduction}
In this paper we introduce the following variant of the Hilbert moduli problem. Let $\X$ be an algebraic stack over a scheme $S$, and suppose that $\X$ admits a good moduli space $\phi\colon\X\to X$. The \emph{good Hilbert functor} is the functor $\GoodHilb_\X\colon\sch_S^{\operatorname{op}}\to\set$ that sends an $S$-scheme $T$ to the set of closed substacks $\Z\subseteq\X_T$ that have proper good moduli spaces. That is, we consider the set of closed substacks $\Z$  fitting into the the following commutative diagram 
\[\xymatrix{\Z\ar@{^(->}[r]\ar[d]&\X_T\ar[d]^{\phi_T}\\Z\ar@{^(->}[r]\ar[dr]&X_T\ar[d]\\&T}\]
where 
\begin{itemize}
\item $\Z\to T$ is flat and finitely presented, and 
\item $Z\to T$ is proper, where $Z:=\phi_T(\Z)$ is the good moduli space of $\Z$. 
\end{itemize}
The main result of this paper is then the following.
\begin{thmx}\label{thm:main}
Let $S$ be a scheme of finite type over a field $k$, and let $\X$ be an algebraic stack of finite type over $S$. Suppose that $\X$ has affine diagonal, has the resolution property, and admits a good moduli space $\X\to X$, such that $X\to S$ is separated. Then, the functor $\GoodHilb_\X$ is an algebraic space that is locally of finite presentation over $S$.
\end{thmx}
The assumptions of the theorem implies that $\X$ is quasi-compact, quasi-separated, and noetherian. A noetherian algebraic stack $\X$ has the \emph{resolution property} if every coherent sheaf $\fF$ on $\X$ has a surjection $\E\twoheadrightarrow\fF$ from a locally free sheaf~$\E$ \cite{MR2108211}. 

\noindent\textbf{Background.} The Hilbert moduli problem seeks to parametrize all closed subschemes of projective space. More generally, given an algebraic stack $\X$ over a scheme $S$, the Hilbert functor $\Hilb_\X\colon\sch_S^{\operatorname{op}}\to\set$ sends an $S$-scheme $T$ to the set of  closed substacks $\Z\subseteq\X_T$ such that the composition $\Z\too\X_T\to T$ is flat, finitely presented, and proper.
The Hilbert functor is an algebraic space when $\X\to S$ is separated and locally of finite presentation \cite{MR2183251}. On the other hand, the Hilbert functor is never algebraic if $\X\to S$ is non-seperated \cite{MR2369042}. See also \cite{MR3148551,MR3359029}. The reason that the algebraicity fails is that formal GAGA does not hold when $\X\to S$ is non-separated. In fact, Grothendieck's existence theorem says that formal GAGA holds for separated morphisms $\X\to\spec(A)$ of finite~type. 

When $\X=U$ is a scheme, many variations of the Hilbert functor have been studied in great detail. In particular, when $U=\spec(R)$ is affine we have the multigraded Hilbert scheme \cite{MR2073194} and the invariant Hilbert scheme \cite{MR2092127}, \cite{MR3184162}. The multigraded case studies closed subschemes of an affine space whose defining ideals have a given Hilbert function with respect to a grading by an abelian group. The invariant Hilbert scheme is a generalization of this and parametrizes closed subschemes that are invariant under an action of a linearly reductive group $G$. By considering the quotient stack of $U$ by~$G$, we have that every $G$-invariant closed subscheme $V\subseteq U$ corresponds to a closed substack $\Z\subseteq[U/G]$. In fact, there is a cartesian square
\[\xymatrix{V\ar[d]\ar[r]&U\ar[d]\\\Z\ar[r]&[U/G]}\]
and $\Z=[V/G]$. 
 Thus, parametrizing invariant closed subschemes $V$ is equivalent with parametrizing closed substacks $\Z\subseteq[U/G]$.
However, when the group $G$ is non-proper, then ${[U/G]\to\spec(A)}$ is non-separated, so the Hilbert functor $\Hilb_{[U/G]}$ is not algebraic. What will save us is that $[U/G]$ admits a good moduli space, in the form of the GIT quotient $[U/G]\to {U\git G}$, which is separated over $\spec(A)$. Good moduli space morphisms have many properties similar to those of proper morphisms, and in \cite{MR3350157} the authors showed that formal GAGA holds, under certain conditions, when $\X\to\spec(A)$ is a good moduli space. That setting is not sufficiently general for us, but we will generalize their result.
\\

\noindent\textbf{This paper.} We introduce the good Hilbert functor which is similar to the classical one. Indeed, there is an inclusion $\Hilb_\X\subseteq\GoodHilb_\X$ of functors, with equality if $\X\to S$ is proper.



The classical Hilbert functor is not algebraic when $\X\to S$ is non-separated, but in Theorem~\ref{thm:main} we show instead that the good Hilbert functor is algebraic under some other assumptions on $\X$. The main problem needed to be worked out in order to show this is a version of formal GAGA for algebraic stacks that admit good moduli spaces. 
For instance, we show the following.
\begin{thmx}\label{thm:coh}
Let $A$ be a complete local noetherian ring, and let $\X$ be a noetherian algebraic stack over $\spec(A)$. Suppose also that $\X$ has the resolution property and that $\X$ admits a good moduli space $X$ that is proper over $\spec(A)$. Then, the completion functor ${\Coh(\X)\to\Coh(\widehat{\X})}$ is an equivalence of categories.
\end{thmx}
This result is not sufficiently general for what we require, but in Theorem~\ref{thm:cohpgs} we prove a stronger version of this theorem, where $X\to\spec(A)$ is only separated and of finite type, and we show that there is an equivalence between the coherent sheaves that have supports admitting good moduli spaces that are proper over $\spec(A)$.

Returning to the invariant Hilbert scheme, we consider the quotient stack $\X=[U/G]$, where $U=\spec(A)$ and $G$ is a linearly reductive group. We then recover the invariant Hilbert scheme as part of a stratification of $\GoodHilb_{[U/G]}$. This is explained in Section~\ref{sec:invHilb}.
\\

\noindent\textbf{Acknowledgements.}
I am indebted to David Rydh for all the invaluable discussions and helpful feedback while writing this paper. Moreover, I thank Jack Hall for his many insightful comments. 

\section*{Conventions and notation}
We will always assume that the algebraic stack $\X$ is quasi-compact and quasi-separated. 
Given an $S$-scheme $T$, we write $\X_T=\X\times_ST$ for the base change. If $T=\spec(A)$ is affine, then we also write $\X_A=\X_T$. Given two sheaves $\fF$ and $\fG$, we let $\hom(\fF,\fG)$ denote the set of morphisms from $\fF$ to $\fG$, and we let $\shom(\fF,\fG)$ denote the sheaf $T\mapsto\hom(\fF|_T,\fG|_T)$. 

The paper \cite{MR3237451} follows the conventions of \cite{MR1771927} and assumes that all stacks have quasi-compact and \emph{separated} diagonal. As all our results will build on this paper, we will make the same assumptions. However, the results of \cite{MR3237451} should remain true without the separatedness assumption, and the same would then hold here. Regardless, many results in Section~\ref{sec:main}, in particular Theorem~\ref{thm:main}, will assume that the stack has even affine diagonal. 

\section{Good moduli spaces}\label{sec:gms}
Good moduli spaces were introduced by Alper in \cite{MR3237451}. A good moduli space of an algebraic stack $\X$ over a scheme $S$ is a quasi-compact and quasi-separated morphism $\phi\colon\X\to X$, where $X$ is an algebraic space over $S$, such that
\begin{enumerate}
\item the natural map $\O_X\to\phi_\ast\O_\X$ is an isomorphism, and
\item the functor $\phi_\ast\colon\QCoh(\X)\to\QCoh(X)$ on quasi-coherent sheaves is exact.
\end{enumerate}
Good moduli spaces are generalizations of GIT-quotients as explained in op.\,cit. Moreover, even though good moduli space morphisms are generally non-separated, they have many properties similar to those of proper morphisms. For us, they are important since they behave well with respect to formal GAGA, as will be explained in Sections~\ref{sec:formalgaga1} and~\ref{sec:formalgaga2}. 

Assuming that $\X$ and $S$ are noetherian we 
 here list a few results concerning the good moduli space of $\X$ that will be used throughout the text. The references to these results are all, except number (7), with respect to op.\,cit. 
\begin{enumerate}
\item 
The good moduli space $X$ is unique up to canonical isomorphism [Theorem~6.6].
\item The map $\phi$ is surjective and universally closed [Theorem~4.16(i+ii)].
\item If $T$ is an $S$-scheme, then the base-change $\phi_T\colon \X_T\to X_T$ is a good moduli space [Proposition~4.7(i)].
\item If $\Z$ is a closed substack of $\X$, then $\Z\to Z:=\phi(\Z)$ is a good moduli space [Lemma~4.14].
\item 
The push-forward functor $\phi_\ast$ takes coherent sheaves on $\X$ to coherent sheaves on~$X$ [Theorem~4.16(x)].
\item The projection morphism $\phi_\ast\fF\otimes\fG\to \phi_\ast(\fF\otimes\phi^\ast\fG)$ is an isomorphism for all quasi-coherent sheaves $\fF\in\QCoh(\X)$ and $\fG\in\QCoh(X)$ [Proposition~4.5].
\item If $\X\to S$ is of finite type, then $X\to S$ is of finite type \cite[Theorem~6.3.3]{MR3272912}.
\end{enumerate}

\section{Artin's criteria for algebraicity}\label{sec:artin}
In \cite{MR0260746}, Artin gave criteria for when a functor $\sch_S^{\operatorname{op}}\to\set$ is an algebraic space. He later extended these results in \cite{MR0399094} to give a criterion for when a category fibered in groupoids over $\sch_S$ is an algebraic stack. This criterion was studied in \cite{2012arXiv1206.4182H}, where a more streamlined version was presented. We state this criterion for the sake of completeness. 
\begin{theorem}[{\cite[Theorem~A]{2012arXiv1206.4182H}}]\label{thm:artin}
Let $\C$ be a category fibered in groupoids over $\sch_S$, where $S$ is excellent (see Remark~\ref{rem:exc}). 
Then, $\C$ is an algebraic stack, locally of finite~presentation over $S$, if and only if the following conditions are satisfied.
\begin{enumerate}
\item \emph{[Stack]} $\C$ is a stack.
\item \emph{[Limit preservation]} For any inverse system of affine $S$-schemes $\{\spec A_j\}_{j\in J}$ with $\varprojlim\kern-0.15em{}_j\kern0.0em\spec(A_j)=\spec(A)$, the natural functor
\[{\textstyle\varinjlim\kern-0.1em{}_j\kern0.1em}\C(\spec A_j)\to\C(\spec A)\]
is an equivalence of categories.
\item \emph{[Homogeneity]} For any diagram of affine $S$-schemes
\[\xymatrix{\spec(B)&\ar[l]\spec(A)\ar[r]^-i&\spec(A'),}\]
where $i$ is a nilpotent closed immersion, the natural functor
\[\C\bigl(\spec(B\times_AA')\bigr)\to \C(\spec B) \times_{\C(\spec A)}\C(\spec A')\]
is an equivalence of categories.
\item \emph{[Effectivity]} For any complete noetherian local ring $(B,\m)$ with an $S$-scheme structure $\spec(B)\to S$, such that the induced morphism $\spec(B/\m)\to S$ is of finite type, the natural functor
\[\C(\spec B)\to\varprojlim\kern-0.2em{}_n\kern0.0em\C(\spec B/\m^n)\]
is an equivalence of categories.
\item \emph{[Conditions on automorphisms and deformations]} For any affine $S$-scheme $T$ that is of finite type over $S$, and $\xi\in\C(T)$, the functors $\operatorname{Aut}_{\C/S}(\xi,-)$ and $\operatorname{Def}_{\C/S}(\xi,-)$ from $\mathbf{QCoh}(T)$ to $\mathbf{Ab}$ are coherent.
\item \emph{[Conditions on obstructions]} For any affine $S$-scheme $T$ that is of finite type over~$S$, and $\xi\in\C(T)$, there exists an integer $n$ and a coherent $n$-step obstruction theory for~$\C$ at $\xi$. 
\end{enumerate}
\end{theorem}

\begin{remark}\label{rem:exc}
We refer to \cite[Definition~12.49]{wedhorn} for a definition of an excellent scheme. A scheme of finite type over a field is excellent, which is the setting we will consider in this paper.
\end{remark}
A category fibered in \emph{setoids} over $\sch_S$ is equivalent to a functor $\sch_S^{\operatorname{op}}\to\set$. Moreover, this equivalence restricts to an equivalence between algebraic stacks fibered in setoids and algebraic spaces. Thus, we have the following immediate consequence.
\begin{corollary}\label{cor:artinspaces}
Let $F\colon\sch_S^{\operatorname{op}}\to\set$ be a functor, where $S$ is excellent. Then, $F$ is an algebraic space that is locally of finite presentation over $S$ if and only if $F$ is a sheaf in the \'etale topology, and the analogous versions of properties 2-6 in Theorem~\ref{thm:artin}, given by replacing \qq{equivalence of categories} with \qq{bijections} and removing the condition on the automorphisms, are satisfied. 
\end{corollary}

We will in Section~\ref{sec:main} show that the good Hilbert functor $\GoodHilb_\X$ satisfies the properties of Corollary~\ref{cor:artinspaces}, and is therefore an algebraic space.
All properties except condition~$4$ turn out to follow from previous work. Indeed, condition 1 is a standard descent argument (both for stacks and for sheaves). Condition~2 is equivalent to $\mathscr{C}$ being locally finitely presented, and can also be shown by standard methods, see e.g.\ \cite[Appendix~B]{MR3272071}.
In \cite{2012arXiv1206.4182H}, Hall also stated the following results that are helpful when verifying conditions~3, 5, and 6. 
\begin{proposition}\label{prop:bootstrap}
Fix an algebraic stack $f\colon\X\to S$, and let $\morj_\X$ be the fibered category consisting of pairs $(T,\Z\to\X_T)$ where $T$ is an $S$-scheme, $\Z\to\X_T$ is a morphism of algebraic stacks, and the composition $\Z\to\X_T\to T$ is flat and locally finitely presented. Then, $\morj_\X$ satisfies property 3. Moreover, given a formally \'etale morphism $\C\to \morj_\X$ of categories fibered in groupoids, then 3 is also satisfied for $\C$. 
\end{proposition}
\begin{proof}
This is proved in \cite{2012arXiv1206.4182H}. More specifically, it is a combination of Lemma~9.3, Lemma~A.6, and Lemma~1.5(9) in op.\,cit. 
\end{proof}

\begin{proposition}\label{prop:cohoftheo}
With the notation of Proposition~\ref{prop:bootstrap}. The automorphisms, deformations, and a $2$-step obstruction theory of an object $(T,g\colon\Z\to\X_T)\in\morj_\X$ are given by certain functors of the form 
\[\ext_{\O_\Z}^n\bigl(\fF,g^\ast f_T^\ast(-)\bigr)\colon\QCoh(T)\to\mathbf{Ab},\] 
where $n=0,1,2$, and $\fF$ is a bounded complex with coherent cohomology. 
Moreover, if ${\C\to\morj_\X}$ is formally \'etale, then these functors also describe automorphisms, deformations and a 2-step obstruction theory for $\C$.
\end{proposition}
\begin{proof}
The first part is proved in Section~9 of \cite{2012arXiv1206.4182H}, while the second part is a combination of Lemma~6.3 and Lemma~6.11 in op.\,cit. 
\end{proof}

We will in our Lemma~\ref{lem:fixarallt} show that the inclusion $\GoodHilb_\X\to\morj_\X$ is formally \'etale, which by the above implies that many of the properties in Artin's criterion hold for the good Hilbert functor. However, showing that the functors in Proposition~\ref{prop:cohoftheo} are coherent is not trivial, but for the good Hilbert functor this follows by earlier results, cf.\,Lemma~\ref{lem:theoroes}.
 Thus, the only remaining problem is showing condition 4, which requires formal GAGA.

\section{Formal GAGA}\label{sec:formalgaga1}
Let $A$ be a complete local noetherian ring with maximal ideal $\m$, and consider an algebraic stack $f\colon\X\to\spec(A)$ of finite type. Then, we let ${\X_0=f^{-1}\bigl(V(\m)\bigr)}$ be the closed substack given by the inverse image of the unique closed point of $\spec(A)$, and let $\I\subseteq\O_\X$ be the corresponding ideal sheaf. Given a coherent sheaf $\fF$ on $\X$, we define \emph{the completion of $\fF$ along $\X_0$} as the sheaf
\[\widehat{\fF}:=\varprojlim\kern-0.2em{}_n\kern0.1em\fF/\I^{n+1}\fF\]
on $\X_{\text{lis-\'et}}$. 
In particular, we define the sheaf (of rings)
\[\O_{\widehat{\X}}:=\widehat{\O_\X}=\varprojlim\kern-0.2em{}_n\kern0.1em\O_\X/\I^{n+1}.\]
Given an algebraic stack $\X$, we define the \emph{completion of $\X$ along $\X_0$} as the ringed topos $\widehat{\X}=(\X_{\text{lis-\'et}},\O_{\widehat{\X}})$, see \cite{conrad_gaga}. This constuction gives a natural completion functor $\Coh(\X)\to\Coh(\widehat{\X})$ defined by $\fF\mapsto\widehat{\fF}$. 
As we are working over a fixed scheme $\spec(A)$ with a unique closed point, the substack $\X_0$ is uniquely defined and we will talk about {completions} without referring to $\X_0$ in the sequel. 

Moreover, letting \[\X_n=\X\times_{\spec(A)}\spec(A/\m^{n+1}),\] we get a sequence $\X_0\to\X_1\to\X_2\to\cdots$ of closed immersions. We define the category of compatible systems $\varprojlim\kern-0.2em{}_n\kern0.0em\Coh(\X_n)$ consisting of sequences $(\fF_n)$ where $\fF_{n}\in\Coh(\X_n)$ and $\fF_{n+1}/\I^{n+1}\fF_{n+1}=\fF_{n}$ for all $n\ge0$. In other words, a compatible system $(\fF_n)$ is a sequence of coherent sheaves such that $\fF_n\otimes_{\O_{\X_n}}\O_{\X_m}=\fF_m$ for all $m\le n$. 
Studying such compatible systems is equivalent to studying coherent sheaves on $\widehat{\X}$, as the following result shows.

\begin{theorem}[{\cite[Theorem~2.3]{conrad_gaga}}]\label{thm:defgagatyp}
The natural functor \[\Coh(\widehat{\X})\to\varprojlim\kern-0.2em{}_n\kern0.1em\Coh(\X_n),\] is an equivalence of categories. 
\end{theorem}
We will call a compatible system $(\fF_n)$ \emph{algebraizable} if there is some $\fF\in\Coh(\X)$ such that $\fF/\I^{n+1}\fF=\fF_n$ for all $n$, that is, if the system $(\fF_n)$ lies in the essential image of the natural functor $\Coh(\X)\to\varprojlim\kern-0.2em{}_n\kern0.0em\Coh(\X_n)$.

Given a coherent sheaf $\fF\in\Coh(\X)$, we have the annihilator ideal sheaf 
\[\Ann_{\O_\X}\!(\fF)=\ker\bigl(\O_\X\to\End_{\O_\X}\!(\fF)\bigr)\] and we define the support of $\fF$ as the closed subset $\supp(\fF)\subseteq|\X|$ defined by $\Ann_{\O_\X}\!(\fF)$. 
\begin{remark}
Equivalently, one can define the support of $\fF$ as the complement of the underlying set of the largest open substack $\mathscr{U}\subseteq\X$ where $\fF$ vanishes. 
\end{remark}
The support will by construction commute with flat base change. That is, given a flat morphism $f\colon\Y\to\X$, then $\supp(f^\ast\fF)=f^{-1}(\supp \fF)$. In fact, similarly to the case for schemes, flatness is not required.
\begin{lemma}\label{lem:comwithbc}
Let $f\colon\Y\to\X$ be a morphism of algebraic stacks and let $\fF$ be a coherent sheaf on $\X$. Then, $\supp(f^\ast\fF)=f^{-1}(\supp\fF)$.
\end{lemma}
\begin{proof}
Let $u\colon U\to \X$ be a smooth surjection with $U$ a scheme, and let $V\to U\times_\X\Y$ be a smooth surjection with $V$ a scheme. Then, we have a commutative diagram
\[\xymatrix{V\ar[r]^v\ar[d]_g&\Y\ar[d]^f\\U\ar[r]^u&\X}\]
where both $u\colon U\to \X$ and $v\colon\ V\to\Y$ are flat and surjective. Supports commute with flat base change, so $\supp(u^\ast\fF)=u^{-1}(\supp\fF)$ and $\supp(v^\ast f^\ast\fF)=v^{-1}(\supp f^\ast\fF)$. \newpage Moreover, supports of coherent sheaves on schemes commute with arbitrary base change \cite[Tag~0BUR]{stacks-project}, so $\supp(g^\ast u^\ast\fF)=g^{-1}(u^\ast\supp\fF)$. Putting these together, we get
\[v^{-1}(\supp f^\ast\fF)=\supp(v^\ast f^\ast\fF)=\supp(g^\ast u^\ast\fF)=g^{-1}u^{-1}(\supp\fF)=v^{-1}f^{-1}(\supp\fF).\]
Since $v$ is surjective it now follows that $\supp(f^\ast\fF)=f^{-1}(\supp\fF)$.
\end{proof}

The support of a coherent sheaf  can be given the structure of a closed substack by giving it the reduced structure, and we let $\Cohps(\X)$ denote the full subcategory of $\Cohps(\X)$ consisting of the coherent sheaves with proper support. That is, we let $\Cohps(\X)$ consist of those~$\fF$ for which $\supp(\fF)\to\spec(A)$ is proper. 

\begin{remark}
The choice of stack structure on the support is irrelevant for asking when $\supp(\fF)\to\spec(A)$ is proper. Indeed, a morphism being separated and universally closed depends on the underlying topological spaces, and being of finite type is automatic here.
\end{remark}
We let $\Cohps(\widehat{\X}):=\varprojlim\kern-0.2em{}_n\kern0.1em\Cohps(\X_n)$. 
A classical result within the theory of formal GAGA is Groth\-endieck's existence theorem, which states that the completion functor\[{\Cohps(\X)\to\Cohps(\widehat{\X})}\] is an equivalence of categories if $\X\to\spec(A)$ is separated and of finite type. In particular, if ${\X\to\spec(A)}$ is proper, then there is an equivalence $\Coh(\X)\to\Coh(\widehat{\X})$. 
This was originally proved for schemes by Grothendieck \cite{MR0217085}, and for algebraic spaces by Knutson \cite{MR0302647}. Later, this was generalized to Deligne-Mumford stacks by Olsson and Starr in \cite{MR2007396} and to algebraic stacks by Olsson in \cite{MR2183251}, \cite[Appendix~A]{MR2239345}.

In \cite{MR3350157} a version of Grothendieck's existence theorem was proved when ${\X\to\spec(A)}$ is not separated, but instead when $\spec(A)$ is a good moduli space of $\X$. More specifically, they showed that if $\X$ is a noetherian algebraic stack with the resolution property, and $\X\to\spec(A)$ is a good moduli space, then there is an equivalence $\Coh(\X)\to\Coh(\widehat{\X})$. In that paper it was also remarked that this could be generalized to when $\X$ admits a good moduli space $\X\to X$, such that $X\to\spec(A)$ is separated and of finite type. 
We will make this claim precise in the next section. Before that we state a useful result which is a special case of more general facts explained in \cite{MR3350157}. 

\begin{lemma}[{\cite[Lemma~3.2-Remark~3.5]{MR3350157}}]\label{lem:exactnesshet}\label{lem:generellt}
Let $\X$ be a noetherian algebraic stack. Then, the completion functor $\Coh(\X)\to\Coh(\widehat{\X})$ is exact. Moreover, the canonical map \[{\shom_{\O_\X}(\fF,\fG)}^{\wedge}\to\shom_{\O_{\widehat{\X}}}(\widehat{\fF},\widehat{\fG})\] is an isomorphism for any $\fF,\fG\in\Coh(\X)$.
\end{lemma}

\section{Formal GAGA  for good moduli spaces}
\label{sec:formalgaga2}
Using the setup of the previous section, we will now consider a noetherian algebraic stack $\X\to\spec(A)$ that has the resolution property, and that admits a good moduli space $\phi\colon\X\to X$, where $X\to\spec(A)$ is separated and of finite type. By Section~\ref{sec:gms}, the pushforward and pullback of $\phi$ restricts to functors $\phi_\ast\colon\Coh(\X)\to\Coh(X)$ and $\phi^\ast\colon\Coh(X)\to\Coh(\X)$ of coherent sheaves. 

For $\X_n=\X\times_{\spec(A)}\spec(A/\m^{n+1})$, we have that $\phi_n\colon\X_n\to X_n:=\phi(\X_n)$ is a good moduli space for any $n$. With some abuse of notation we will in the sequel write $\phi_n=\phi$. Letting $I$ denote the ideal sheaf that defines $X_0$ in $X$, it follows that the ideal $I^{n+1}$ then defines $X_n$ in $X$ for all $n$. 

\begin{lemma}\label{lem:alg}
Let $\phi\colon\X\to X$ be a good moduli space. Given a compatible system $(\fF_n)$ of coherent sheaves on~$\X$, 
we have that
$(\phi_\ast\fF_n)$ is a compatible system of coherent sheaves on~$X$. Moreover, given $\fF\in\Coh(\X)$, we have that $(\phi_\ast\fF)\otimes\O_{X_n}=\phi_\ast(\fF\otimes\O_{\X_n})$. 
\end{lemma}
\begin{proof}
Let $(\fF_n)$ be a compatible system of coherent sheaf on $\X_n$. Then, since the projection morphism is an isomorphism \cite[Proposition~4.5]{MR3237451}, it follows that 
\[(\phi_\ast\fF_n)\otimes\O_{X_m}=\phi_\ast(\fF_n\otimes\phi^\ast\O_{X_m})= \phi_\ast(\fF_n\otimes\O_{\X_m})=\phi_\ast(\fF_m)\]
for $m\le n$. The second statement follows in the same way by \cite[Proposition~4.5]{MR3237451}. 
\end{proof}

When we give the support of a coherent sheaf $\fF\in\Coh(\X)$ a stack structure, we get a good moduli space
\[\supp(\fF)\to \suppp(\fF):=\phi(\supp\fF),\]
and we call this algebraic space the \emph{good support} of $\fF$. We define $\Cohpgs(\X)$ as the full subcategory of $\Coh(\X)$ consisting of the coherent sheaves with  proper good support. That is, we let $\Cohpgs(\X)$ consist of coherent sheaves $\fF$ where $\suppp(\fF)\to\spec(A)$ is proper. There is then an inclusion of categories \[\Cohps(\X)\subseteq\Cohpgs(\X)\subseteq\Coh(\X).\] If $X\to \spec(A)$ is proper, then we have an equality $\Cohpgs(\X)=\Coh(\X)$, and if $\X=X$ is an algebraic space, then there is an equality $\Cohps(X)=\Cohpgs(X)$.

\begin{lemma}\label{lem:serre}\label{lem:hoppas}
Let $\phi\colon\X\to X$ be a good moduli space. 
\begin{enumerate}
\item For $\fF,\fG\in\Cohpgs(\X)$ we have:
\begin{enumerate}
\item  if $f\colon\fF\to\fG$ is a morphism,
then $\ker(f),\coker(f)\in\Cohpgs(\X)$.
\item $\fF\otimes_{\O_\X}\kern-0.25em\fG,\kern0.2em \shom_{\O_{\X}}\kern-0.1em(\fF,\fG)\in\Cohpgs(\X)$.
\end{enumerate}
\item $\phi_\ast\colon\Coh(\X)\to\Coh(X)$ restricts to a functor $\phi_\ast\colon\Cohpgs(\X)\to\Cohps(X)$.
\item $\phi^\ast\colon\Coh(X)\to\Coh(\X)$ restricts to a functor $\phi^\ast\colon\Cohps(X)\to\Cohpgs(\X)$.
\end{enumerate}
\end{lemma}
\begin{proof} These all follow by general results on the support.
\begin{enumerate}
\item Good moduli spaces are universally closed, and closed immersions are proper. Moreover, compositions of proper are proper. Using these results we have:
\begin{enumerate}
\item\label{item:a} $\supp(\ker f)$ is closed in $\supp(\fF)$, so $\suppp(\ker f)$ is closed in $\suppp(\fF)$. Thus, the composition $\suppp(\ker f)\too\suppp(\fF)\to\spec(A)$ is proper. The case of $\supp(\coker f)$ follows analogously as it is closed in $\supp(\fG)$. 
\item $\supp(\fF\otimes_{\O_\X}\!\fG)$ and $\supp\bigl(\shom_{\O_{\X}}(\fF,\fG)\bigr)$ are closed in $\supp(\fF)\cap\supp(\fG)$. In particular, they are closed in $\supp(\fF)$, and the result follows as in (a).
\end{enumerate}
\item For any $\fF\in\Cohpgs(\X)$ we write $U=X\setminus\phi(\supp\fF)$ and $\mathscr{U}=\X\setminus\supp(\fF)$. 
Given $Y\to U$ in the \'etale site of $U$ we have that \[\phi_\ast\fF(Y)=\fF(Y\times_X\X)=\fF(Y\times_U\mathscr{U})=0,\] 
which implies that $(\phi_\ast\fF)|_U=0$. The support of $\phi_\ast\fF$ is the complement of the largest open subalgebraic space on which $\phi_\ast\fF$ vanishes, which implies that \[\supp(\phi_\ast\fF)\subseteq\phi(\supp\fF)=\suppp(\fF).\] 
Thus, the composition $\supp(\phi_\ast\fF)\hookrightarrow\suppp(\fF)\to\spec(A)$ is proper. 

\item Take $G\in\Cohps(X)$. By Lemma~\ref{lem:comwithbc}, we have that $\supp\phi^\ast G=\phi^{-1}(\supp G)$. Thus, 
\[\suppp(\phi^\ast G)=\phi(\supp\phi^\ast G)= \phi\bigl(\phi^{-1}(\supp G)\bigr)=\supp(G),\] 
and the result follows. 
\qedhere\end{enumerate}
\end{proof}

In the sequel, we will write $\Cohpgs(\widehat{\X}):=\varprojlim\kern-0.2em{}_n\kern0.1em\Cohpgs(\X_n)$.
\begin{proposition}\label{prop:gmsff2}
Suppose that $\X$ is a noetherian algebraic stack that admits a good moduli space $\phi\colon\X\to X$, such that $X\to \spec(A)$ is separated and of finite type. Then the completion functor ${\Cohpgs(\X)\to\Cohpgs(\widehat{\X})}$ is fully faithful.
\end{proposition}
\begin{proof}
Given any $\fF,\fG\in\Cohpgs(\X)$, we need to show that we have an equality
\[\hom_{\O_\X}(\fF,\fG)=\hom_{\O_{\widehat{\X}}}(\widehat{\fF},\widehat{\fG}).\]
With $\fH=\shom_{\O_\X}(\fF,\fG)\in\Cohpgs(\X)$, we note that we have an equality of sets 
\[\hom_{\O_\X}(\fF,\fG)=\Gamma(\fH)=\Gamma(\phi_\ast\fH)=\hom_{\O_X}(\O_X,\phi_\ast\fH).\] 
Since  $X\to S$ is separated and of finite type, we have that $\Cohps(X)\to\Cohps(\widehat{X})$ is an equivalence \cite[Theorem~6.3]{MR0302647}. In particular, $\Cohps(X)\to\Cohps(\widehat{X})$ is fully faithful, so it follows that
\[\hom_{\O_\X}(\fF,\fG)=
\hom_{\O_X}(\O_X,\phi_\ast\fH)=\hom_{\O_{\widehat{X}}}\bigl(\O_{\widehat{X}},\widehat{\phi_\ast\fH}\bigr)=\Gamma\bigl(\widehat{\phi_\ast\fH}\bigr).\]
As the projection morphism is an isomorphism for good moduli spaces \cite[Proposition~4.5]{MR3237451} we have that 
$\phi_\ast(\fH\otimes\O_{\X_n})=(\phi_\ast\fH)\otimes\O_{X_n}$ for all $n$. 
Thus, 
\[\widehat{\phi_\ast}\widehat{\fH}=\varprojlim\kern-0.2em{}_n\kern0.1em \phi_\ast(\fH\otimes\O_{\X_n})= \varprojlim\kern-0.2em{}_n\kern0.1em(\phi_\ast\fH\otimes\O_{X_n})=\widehat{\phi_\ast\fH},\]
where $\widehat{\phi_\ast}\colon\Coh(\widehat{\X})\to\Coh(\widehat{X})$ denotes the induced map.
This implies that
\[\hom_{\O_\X}(\fF,\fG)=\Gamma\bigl(\widehat{\phi_\ast\fH}\bigr)= \Gamma\bigl(\widehat{\phi_\ast}\widehat{\fH}\bigr)=\Gamma\bigl(\widehat{\fH}\bigr)= \Gamma\bigl(\shom_{\O_\X}(\fF,\fG)^\wedge\bigr).\]
 Lemma~\ref{lem:generellt} now gives that 
\begin{align*}
\hom_{\O_\X}(\fF,\fG)&=
\Gamma\bigl(\shom_{\O_\X}(\fF,\fG)^\wedge\bigr)=\Gamma\bigl(\shom_{\O_{\widehat{\X}}}(\widehat{\fF},\widehat{\fG})\bigr)=\hom_{\O_{\widehat{\X}}}(\widehat{\fF},\widehat{\fG}).
\qedhere\end{align*}
\end{proof}

\begin{theorem}\label{thm:cohpgs}
Let $A$ be a complete local noetherian ring, and let $\X$ be a noetherian algebraic stack over $\spec(A)$. Suppose that $\X$ has the resolution property and that $\X$ admits a good moduli space $\phi\colon\X\to X$, where $X\to \spec(A)$ is separated and of finite type. Then $\Cohpgs(\X)\to\Cohpgs(\widehat{\X})$ is an equivalence of categories.
\end{theorem}
\begin{proof}
From Proposition~\ref{prop:gmsff2} we have that the functor $\Cohpgs(\X)\to\Cohpgs(\widehat{\X})$ is fully faithful.
We will now show that this functor is also essentially surjective. That is, we will show that every $\mathfrak{F}\in\Cohpgs(\widehat{\X})$ is algebraizable. Such a coherent sheaf $\ffF$ is equivalent to a compatible system $(\fF_n)$ with $\fF_n\in\Cohpgs(\X_n)$ for each $n$. 

Let $j_0\colon\X_0\to\X$ denote the inclusion. Then $j_\ast\fF_0$ is a sheaf on $\X$ and by the resolution property there is a locally free sheaf $\E\in\Coh(\X)$ that surjects onto $j_\ast\fF_0$. Letting ${\E_n=\E\otimes\O_{\X_n}}$, we see for any $m\le n$ that 
\[\shom(\E_n,\fF_n)\otimes\O_{\X_m}=(\E_n^\ast\otimes\fF_n)\otimes\O_{\X_m}= \E_m^\ast\otimes\fF_m=\shom(\E_m,\fF_m).\]
Thus, $\bigl(\shom(\E_n,\fF_n)\bigr)$ is a compatible system of coherent sheaves on $(\X_n)$, and Lemma~\ref{lem:alg} then implies that $\bigl(\phi_\ast\shom(\E_n,\fF_n)\bigr)$ is a compatible system on $(X_n)$. As $X\to S$ is separated and of finite type this system is algebraizable \cite[Theorem~6.3]{MR0302647}. That is, there is a coherent sheaf $G\in\Cohps(X)$ such that $G\otimes\O_{X_n}=\phi_\ast\shom(\E_n,\fF_n)$ for all $n$. 

We now let $\fG=\phi^\ast G$, so that 
\[\fG_n:=\fG\otimes\O_{\X_n}=\phi^\ast(G\otimes\O_{X_n})=\phi^\ast\phi_\ast\shom(\E_n,\fF_n),\] 
and we consider the compatible system $(\E_n\otimes\fG_n)$ on $(\X_n)$. 
By combining the pullback-push\-forward-adjointness with the hom-tensor-adjointness we get canonical maps 
\[\E_n\otimes\fG_n=\E_n\otimes\phi^\ast\phi_\ast\shom(\E_n,\fF_n)\to\fF_n\]
for all $n$ \cite[1.6.2]{2013arXiv1306.5418G}. 

We started by choosing a surjection $\E\twoheadrightarrow j_\ast\fF_0$ giving a surjection $s\colon\E_0\twoheadrightarrow\fF_0=j^\ast j_\ast\fF_0$. It follows that $\E_0\otimes\fG_0\twoheadrightarrow\fF_0$ is surjective. Indeed, $\fG_0$ has a global section given by $s$, and the induced composition $\E_0\to\E_0\otimes\fG_0\to\fF_0$ equals the surjection $s\colon\E_0\twoheadrightarrow\fF_0$. By Nakayama's lemma it therefore follows that $\E_n\otimes\fG_n\twoheadrightarrow\fF_n$ is surjective for all $n$. Note that $\E_n\otimes\fG_n=\E_n\otimes\phi^\ast\phi_\ast\shom(\E_n,\fF_n)$ is a coherent sheaf with proper good support by Lemma~\ref{lem:hoppas}. 
As  
\[\varprojlim\kern-0.2em{}_n\kern0.1em (\E_n\otimes\fG_n)= \varprojlim\kern-0.2em{}_n\kern0.1em \bigl((\E\otimes\fG)\otimes\O_{\X_n}\bigr) = \widehat{\E\otimes\fG},\]
we get a surjection $\widehat{\fG'}\to\mathfrak{F}$ from an algebraizable sheaf, where $\fG'=\E\otimes\fG$. Applying the same approach to the kernel of this surjection we get a presentation
\[\widehat{\fG''}\to\widehat{\fG'}\to\mathfrak{F}\to0.\]
By the full faithfulness of $\Cohpgs(\X)\to\Cohpgs(\widehat{\X})$ we have that the morphism $\widehat{\fG''}\to\widehat{\fG'}$ corresponds to a morphism $\fG''\to\fG'$, and we let $\fF=\coker(\fG''\to\fG')$. The exactness of the completion functor stated in Lemma~\ref{lem:generellt} now shows that
\[\mathfrak{F}=\coker\bigl(\widehat{\fG''}\to\widehat{\fG'}\bigr)= \coker\bigl(\fG''\to\fG'\bigr)^\wedge=\widehat{\fF}\] 
is algebraizable.
\end{proof}

\begin{remark}
When $X\to\spec(A)$ is proper we have that $\Cohpgs(\X)=\Coh(\X)$. Thus, Theorem~\ref{thm:coh} from the introduction follows as an immediate corollary of Theorem~\ref{thm:cohpgs}.

\end{remark}

\section{The good Hilbert functor is an algebraic space}\label{sec:main}
In this section we give a definition of the good Hilbert functor that we introduced earlier, and give a proof of Theorem~\ref{thm:main}. 
\begin{definition}\label{def:main}
Let $\X$ be an algebraic stack over a scheme $S$, and suppose that $\X$ admits a good moduli space $\phi\colon\X\to X$, where $X\to S$ is separated. The \emph{good Hilbert functor} is the functor $\GoodHilb_\X\colon\sch_S^{\operatorname{op}}\to\set$ that sends an $S$-scheme $T$ to the set of closed substacks ${\Z\subseteq\X_T}$ such that the composition $\Z\to\X_T\to T$ is flat and finitely presented, and the composition $Z:=\phi_T(\Z)\to X_T\to T$ is proper. 
Given a morphism $f\colon T'\to T$ of $S$-schemes, we define 
\[\GoodHilb_\X(f)\colon\GoodHilb_\X(T)\to\GoodHilb_\X(T')\]
by $(\Z\to\X_{T})\mapsto(\Z\times_{T}T'\to\X_{T'})$. 
\end{definition}
Throughout this section we will assume that $\X$ satisfies the assumptions stated in the definition above. We will now apply Artin's criterion that we discussed in Section~\ref{sec:artin} to show that the good Hilbert functor is an algebraic space.  
\begin{lemma}\label{lem:ssistas}
The functor $H_\X\colon\sch_S^{\operatorname{op}}\to\set$ defined by $T\mapsto\{\text{closed substacks }\Z\subseteq\X_T\}$ is a sheaf in the fpqc topology.
\end{lemma}
\begin{proof}
%
Let $\{T_i\to T\}$ be an fpqc cover of an $S$-scheme $T$, and write $T_{ij}=T_i\times_TT_j$.
 We need to show the exactness of the diagram
\[\textstyle H_\X(T)\to\prod_i H_\X({T_i})\rightrightarrows\prod_{i,j} H_\X({T_{ij}}).\]
Let $U\to\X$ be a presentation where $U$ is a scheme, and write $R=U\times_{\X}U$. 
Then we have a commutative diagram
\[\xymatrix{
H_{R}(T)\ar@{^(->}[r]&
\prod_iH_{R}(T_i)\ar@<-.5ex>[r] \ar@<.5ex>[r]& 
\prod_{i,j}H_{R}(T_{ij})\\
H_U(T)\ar@<-.5ex>[u] \ar@<.5ex>[u]\ar@{^(->}[r]&
\prod_iH_{U}(T_i)\ar@<-.5ex>[u] \ar@<.5ex>[u]\ar@<-.5ex>[r] \ar@<.5ex>[r]&
\prod_{i,j}H_{U}(T_{ij})\ar@<-.5ex>[u] \ar@<.5ex>[u]\\
H_\X(T)\ar@{^(->}[u]\ar[r]&
\prod_iH_\X({T_i})\ar@<-.5ex>[r] \ar@<.5ex>[r]\ar@{^(->}[u]&
\prod_{i,j}H_\X({T_{ij}})\ar@{^(->}[u]}\]
where all columns are exact by the construction of quasi-coherent sheaves on stacks, and the top two rows are exact by \cite[Tag~023T]{stacks-project}. 
That the bottom row is exact now follows by a simple diagram chase. 
%
\end{proof}

\begin{lemma}\label{lem:sheaf}
The functor $\GoodHilb_\X$ is a sheaf in the fpqc topology.
\end{lemma}
\begin{proof}
As the properties of being flat, finitely presented, and proper are fpqc local on the base \cite[Tags~041W, 041V, 0422]{stacks-project}, the result follows from Lemma~\ref{lem:ssistas}.
\end{proof}

\begin{lemma}\label{lem:sista}
Let $\{\spec A_j\}_{j\in J}$ be an inverse system of affine schemes over $S$ with limit ${\varprojlim\kern-0.15em{}_j\kern0.0em\spec(A_j)=\spec(A)}$. Then, the natural map
\[{\Psi\colon\textstyle\varinjlim\kern-0.15em{}_j\kern0.0em}\GoodHilb_\X(\spec A_j)\to\GoodHilb_\X(\spec A)\]
is a bijection of sets.
\end{lemma}
\begin{proof}
Given a stack $\Z_i\to\spec(A_i)$, we write $\Z_i|_{A_j}:=\Z_i\times_{\spec(A_i)}\spec(A_j)$ for all $j\ge i$. A basic fact of direct limits is then that the set $\varinjlim\kern-0.15em{}_j\kern0.0em\GoodHilb_\X(\spec A_j)$ is in bijection with the set of equivalence classes of \[\{(A_j,\Z_j)\mid \Z_j\in\GoodHilb_\X(\spec A_j), j\in J\}\] under the relation $(A_i,\Z_i)\sim(A_j,\Z_j)$ if $\Z_i|_{A_m}=\Z_j|_{A_m}$ for some $m\ge i,j$. 
The map $\Psi$ sends an equivalence class $[(A_j,\Z_j)]$ to ${\Z_j\times_{\spec(A_j)}\spec(A)\in\GoodHilb_\X(\spec A)}$. 

On the other hand, for a closed substack $\Z\in\GoodHilb_\X(\spec A)$ we have by \cite[Appendix~B, Proposition~(B.2)]{MR3272071} that there is an index $i\in J$ with an algebraic stack $\Z_i$ of finite presentation over $\spec(A_i)$, together with a morphism $\Z_i\to\X_{A_i}$, such that ${\Z_i\times_{\spec(A_i)}\spec(A)=\Z}$. 
The stack $\Z_i$ might not be an element of $\GoodHilb_\X(\spec A_i)$, but [op.\,cit., Proposition~(B.3)] shows that ${\Z_j:=\Z_i|_{A_j}}$ is a closed substack of $\X_{A_j}$ that is flat and finitely presented over $\spec(A_j)$ for $j\gg i$. Moreover, fixing $j$ we have that $\Z_j$ has a good moduli space $Z_j$ and ${\Z_m=\Z_j|_{A_m}}$ has a good moduli space $Z_m=Z_j|_{A_m}$ for all $m\ge j$. By applying [op.\,cit., Proposition~(B.3)] again, we have that $Z_m\to\spec(A_m)$ is proper for $m\gg j$. Thus, ${\Z_m\in\GoodHilb_\X(\spec A_m)}$ for $m\gg i$, and the map $\Psi$ is invertible with inverse $\Z\mapsto[(A_m,\Z_m)]$. 
\end{proof}

\begin{lemma}\label{lem:fixarallt}
Let $\morj_\X$ denote the fibered category defined in Proposition~\ref{prop:bootstrap}. The inclusion $\GoodHilb_\X\to \morj_\X$ is formally \'etale, that is, for every commutative ring $A$ with nilpotent ideal $I$, and for all  morphisms $\spec(A/I)\to\GoodHilb_\X$ and $\spec(A)\to\morj_\X$ such that the diagram
\[\xymatrix{\spec(A/I)\ar[r]\ar@{^(->}[d]&\GoodHilb_\X\ar[d]\\
\spec(A)\ar[r]\ar@{-->}[ur]&\morj_\X}\]
commutes, there exists a unique morphism $\spec(A)\dashrightarrow\GoodHilb_\X$ filling in the diagram.
\end{lemma}
\begin{proof}
Given a commutative diagram as above, we get a commutative diagram
\[\xymatrix@R=0.5em@C=-0.1em{
&					\X_{A/I}\ar[rr]\ar[dd]&&			\X_A\ar[dd]\\
\Z_1\ar[rr]\ar[dd]\ar[ur]&		&			\Z_2\ar[dd]\ar[ur]\\
&X_{A/I}\ar[rr]\ar[dddl]&&X_A\ar[dddl]\\
Z_1\ar[dd]\ar[ur]\ar[rr]&		&	Z_2\ar[dd]\ar[ur]		\\\\
\spec(A/I)\ar[rr]&		&			\spec(A)}\]
where the top, front, back, and bottom squares are all cartesian. 
Here we have written $Z_1=\phi_{A/I}(\Z_1)$ and $Z_2=\phi_A(\Z_2)$. By definition $\Z_1\to\X_{A/I}$ is a closed embedding, $\Z_1\to Z_1$ a good moduli space, $Z_1\to\spec(A/I)$ is proper, and both $\Z_1\to\spec(A/I)$ and $\Z_2\to\spec(A)$ are flat and locally finitely presented. We need to show that $\Z_2\to\X_A$ is a closed embedding and that $\Z_2\to Z_2$ is a good moduli space such that $Z_2\to\spec(A)$ is proper. 

That $\Z_2\to\X_A$ is a closed embedding is a local property on the target so we can assume that $\X_A$ is a scheme and apply  \cite[Tag~09ZW]{stacks-project} on the top square from which the result follows. 
Given that $\Z_2\to\X_A$ is a closed embedding it then follows that $\Z_2\to Z_2$ is a good moduli space. 
By considering the lower front cartesian square it now follows that $Z_2\to \spec(A)$ is proper \cite[Tag~09ZZ]{stacks-project}.
%
\end{proof}

\begin{lemma}\label{lem:theoroes}
Suppose that $\X$ is of finite type over a field, and that $\X$ has affine diagonal. Then the good Hilbert functor $\GoodHilb_\X$ has coherent automorphism, deformation, and obstruction theories. 
\end{lemma}
\begin{proof}
As the inclusion $\GoodHilb_\X\to \morj_\X$ is formally \'etale by Lemma~\ref{lem:fixarallt}, it follows from Proposition~\ref{prop:cohoftheo} that the automorphisms, deformations, and obstructions of an object ${(\Z\to\X_T)\in\GoodHilb_\X(T)}$ are of the form $\ext_{\O_\Z}^n\!\bigl(\fF,g^\ast f_T^\ast(-)\bigr)$ for certain bounded complexes~$\fF$ with coherent cohomology. By \cite[Proposition~12.14]{MR3237451}, any closed point of $\X$ has linearly reductive stabilizer, so we can apply \cite[Theorem~2.26]{2015arXiv150406467A} which says that the derived category $\operatorname{D}_{\QCoh}(\X)$ is compactly generated. Thus, the assumptions of \cite[Corollary~4.16]{2014arXiv1405.1887H} are satisfied, which states that the functors $\ext_{\O_\Z}^n\!\bigl(\fF,g^\ast f_T^\ast(-)\bigr)$ are indeed coherent.
\end{proof}


{
\renewcommand{\thethmx}{\ref{thm:main}}
\begin{thmx}
Let $S$ be a scheme of finite type over a field $k$, and let $\X$ be an algebraic stack of finite type over $S$. Suppose that $\X$ has affine diagonal, has the resolution property, and admits a good moduli space $\X\to X$, such that $X\to S$ is separated. Then, the functor $\GoodHilb_\X$ is an algebraic space that is locally of finite presentation over $S$.
\end{thmx}
}

\begin{proof} We summarize our previous results that implies that $\GoodHilb_\X$ satisfies Artin's criterion in the form of Corollary~\ref{cor:artinspaces}. 
\begin{enumerate}
\item As the fpqc topology is finer than the \'etale topology, this follows by Lemma~\ref{lem:sheaf}.
\item This is Lemma~\ref{lem:sista}. 
\item Follows by combining Lemma~\ref{lem:fixarallt} with Proposition~\ref{prop:bootstrap}. 
\item This is Theorem~\ref{thm:cohpgs} applied to quotients of the structure sheaf. 
\item This is Lemma~\ref{lem:theoroes}.
\item This is Lemma~\ref{lem:theoroes}.
\qedhere\end{enumerate}
\end{proof}

\begin{remark}
Jack Hall has pointed out that the conclusion of Lemma~\ref{lem:theoroes} remains true when replacing the assumption of $\X$ being of finite type over a field with $\X$ having the resolution property. Moreover, a forthcoming result of Jarod Alper, Jack Hall and David Rydh will show that the conclusion of Lemma~\ref{lem:theoroes} is true without either the assumptions of being over a field or having the resolution property. Thus, Theorem~\ref{thm:main}  can be strengthened by replacing the base scheme of finite type over a field with a general excellent scheme.
\end{remark}


\section{Invariant Hilbert schemes}\label{sec:invHilb}
In this section, we will apply our results on the good Hilbert functor to describe the invariant Hilbert scheme described in \cite{MR3184162}. Our results can however be presented in a  slightly more general setting. All schemes we consider will be noetherian over a field $k$. 

Let $S$ be a scheme of finite type over $k$, and consider a flat affine group scheme $G$ of finite type over $S$. Let $G$ act on an affine scheme $U=\spec(A)$ of finite type over $S$. Then the quotient stack $[U/G]$ is noetherian and of finite type over $S$. Furthermore, a closed subscheme $V$ of $U$ that is invariant under the action of $G$ is equivalent to a closed substack $\Z$ of $[U/G]$. In fact, $\Z=[V/G]$ and $V=\Z\times_{[U/G]}U$. Moreover, $V\to S$ is flat if and only if $[V/G]\to S$ is flat. Thus, parametrizing flat closed invariant subschemes of $U$ is equivalent to parametrizing flat closed substacks of $[U/G]$. 

We call $G$ linearly reductive if the structure morphism $BG\to S$ of the classifying stack of $G$ is a good moduli space \cite[Definition~12.1]{MR3237451}. If $G$ is linearly reductive then we have, by \cite[Theorem~13.2]{MR3237451}, that 
\[\phi\colon [U/G]\to U\git G:=\spec(\pi_\ast\O_{[U/G]})=\spec(A^G)\] 
is a good moduli space, where $\pi\colon[U/G]\to S$ denotes the structure morphism. 
Motivated~by these results, we call the good Hilbert functor $\GoodHilb_{[U/G]}$ the \emph{$G$-invariant Hilbert functor} of~$U$.

\begin{proposition}
With the assumptions above, the $G$-invariant Hilbert functor of $U$ is algebraic. 
\end{proposition}
\begin{proof}
We show that $[U/G]$ satisfies the assumptions of Theorem~\ref{thm:main}. There is a cartesian square
\[\xymatrix{G\times_SU\ar[r]\ar[d]&U\times_SU\ar[d]\\[U/G]\ar[r]&[U/G]\times_S[U/G]}\]
where the vertical maps are faithfully flat and of finite presentation. Since $G$ and $U$ are affine we have that $G\times_SU\to U\times_SU$ is affine. It therefore follows by descent that the diagonal $[U/G]\to[U/G]\times_S[U/G]$ is affine. 
Moreover, $[U/G]$ has the resolution property, see e.g.\ \cite[Theorem~2.1]{MR2108211}. As ${U\git G}\to S$ is affine, it is separated. Thus, Theorem~\ref{thm:main} states that $\GoodHilb_{[U/G]}$ is algebraic.
\end{proof}

From now on, we assume that $S=\spec(k)$. Then $G$ is linearly reductive if and only if any \mbox{finite} dimensional representation of $G$ splits as a direct sum of irreducible representations \cite[Proposition~12.6]{MR3237451}. 
For any $S$-scheme $T$ we can consider a closed substack $[V/G_T]$ of ${[U_T/G_T]=[U/G]\times_ST}$. Letting $\pi$ denote the composition ${[V/G_T]\hookrightarrow [U_T/G_T]\to T}$, we have that 
\[[V/G_T]\to{V\git G_T}:=\spec(\pi_\ast\O_{[V/G_T]})\] is a good moduli space. If $[V/G_T]\in\GoodHilb_{[U/G]}(T)$, then $\pi\colon[V/G_T]\to T$ is flat and finitely presented, and ${V\git G_T}\to T$ is proper. 

Let us study the condition that ${V\git G_T}\to T$ is proper. As $[V/G_T]\to T$ is of finite type by assumption, it follows from \cite[Theorem~6.3.3]{MR3272912} that ${V\git G_T}\to T$ is of finite type. Since ${V\git G_T}\to T$ is affine it is also separated. Thus, the only non-automatic part of this condition is that ${V\git G_T}\to T$ is universally closed. That an affine morphism is universally closed is equivalent to it being integral, which is also equivalent to $\pi_\ast\O_{[V/G]}$ being a finitely generated $\O_T$-module. The condition that $[V/G_T]\to T$ is flat implies that also ${V\git G_T}\to T$ is flat \cite[Theorem~4.16(ix)]{MR3237451}. Thus, we conclude that the ${V\git G_T}\to T$ is proper if and only if $\pi_\ast\O_{[V/G]}$ is a locally free $\O_T$-module of finite rank.

We now study the condition that $\pi\colon[V/G_T]\to T$ is flat. Let $p\colon V\to T$ denote the composition ${V\hookrightarrow U_T\to T}$. Then, $\pi_\ast\O_{[V/G_T]}=(p_\ast\O_V)^G$ is the ring of invariants of $p_\ast\O_V$. Also, the morphism $\pi\colon[V/G_T]\to T$ is flat if and only if $p\colon V\to T$ is flat.  Let $\operatorname{Irr}(G)$ denote the set of isomorphism classes of irreducible representations of $G$. Using the ideas of \cite{MR3184162}, we then have that $p\colon V\to T$ is flat if and only if each sheaf of covariants $\fF_M:=(M^\ast\otimes_{k}p_\ast\O_{V})^G$ is a locally free $\O_T$-module of finite rank for any $M\in\operatorname{Irr}(G)$. 

Given a function $h\colon\operatorname{Irr}(G)\to\N$ we let $\GoodHilbh_{[U/G]}$ denote the functor parametrizing flat closed substacks with Hilbert function~$h$. That is, for any $S$-scheme $T$, we let
\[\GoodHilbh_{[U/G]}(T)=\left\{[V/G_T]\in\GoodHilb_{[U/G]}(T)\,\left\vert\,\begin{aligned}& \text{$\fF_M$ locally free $\O_T$-module of}\\&\text{rank $h(M)$ for all $M\in\operatorname{Irr}(G)$}\end{aligned}\right.\right\}.\]
%
We saw above that an element $[V/G_T]\in\GoodHilb_{[U/G]}(T)$ required that all the corresponding sheaves of covariants were locally free of some finite ranks, so it follows that
\[\GoodHilb_{[U/G]}=\coprod_h\GoodHilbh_{[U/G]},\]
where the disjoint union is taken over all possible functions $h\colon\operatorname{Irr}(G)\to\N$. 
The functor  $\GoodHilbh_{[U/G]}$ is precisely the functor studied in \cite{MR3184162} in the setting of invariant closed subschemes of $U$, and there it was shown that this functor is representable by a scheme called the invariant Hilbert scheme. 
%
\begin{ex}
Suppose that $G=\spec(k[L])$ is diagonalizable, where $L$ is an abelian group, and let $U=\spec(A)$. As is explained in \cite[Example~2.2]{MR3184162}, the irreducible representations of $G$ are in a one-to-one correspondence with the elements of $L$, and the action of $G$ on $U$ is equivalent with a grading $A=\bigoplus_{a\in L}A_a$. For any function $h\colon L\to\N$, we then have that $\GoodHilbh_{[U/G]}$ is the multigraded Hilbert scheme of \cite{MR2073194}. 
\end{ex}

The results of this paper do not give any new results concerning these fundamental objects, but the general framework we have constructed gives possibilities of natural generalizations that we leave for future work. 



\bibliography{references}{}
\bibliographystyle{amsalpha}

\end{document}